\documentclass
{amsart}
\usepackage{amssymb,amsthm,amsmath,amsfonts,euscript,qsymbols,delarray}
\usepackage{latexsym,amstext,amscd}
\usepackage{enumerate,hyperref}
\usepackage{color}

\usepackage{amssymb,amsmath,tikz,amsthm}
\usepackage{enumerate,hyperref}
\usepackage[all]{xypic}
\usepackage{vmargin}
\usepackage{enumitem}
\usepackage[utf8]{inputenc}

\usepackage{graphicx}
\usepackage[all]{xy}
\makeatletter
\@namedef{subjclassname@2020}{%
  \textup{2020} Mathematics Subject Classification}
\makeatother

\newtheorem{theorem}{Theorem}[section]

\newtheorem{proposition}[theorem]{Proposition}
\newtheorem{corollary}[theorem]{Corollary}
\newtheorem{lemma}[theorem]{Lemma}

{Claim}

\theoremstyle{definition}
\newtheorem{remark}[theorem]{Remark}
\newtheorem{definition}[theorem]{Definition}
\newtheorem{fact}[theorem]{Fact}

\newtheorem{question}[theorem]{Question}

\newcommand{\Z}{\mathbb{Z}}
\newcommand{\N}{\mathbb{N}}

\DeclareMathOperator{\res}{\restriction}

\usepackage{vmargin}
\setmarginsrb{29mm}{10mm}{29mm}{18mm}
   {29mm}{7mm}{29mm}{18mm}
\catcode`\@=12

\def\Z{{\mathbb Z}}
\def\N{{\mathbb N}}

\def\nbd{neighbourhood}

\input xy  \xyoption{all}

\begin{document}

\title[Lattice of Group Topologies]{On The Lattice of Group Topologies}

\author{Dekui Peng}
\address[D. Peng]
	{\hfill\break Institute of Mathematics,
		\hfill\break Nanjing Normal University, 210024,
		\hfill\break China}
\email{pengdk10@lzu.edu.cn}

\begin{abstract}
For an infinite group $G$, the poset $\mathcal{L}_G$ of group topologies constitutes a complete lattice. Although $\mathcal{L}_G$ is modular when $G$ is abelian, this property fails to persist for nilpotent groups. Extending Arnautov's 2010 work on the semi-modularity of $\mathcal{L}_G$ for nilpotent groups, we present an alternative proof with enhanced structural clarity. Additionally, we resolve two open questions from the Kourovka Notebook regarding lattice-theoretic properties of $\mathcal{L}_G$: (1) explicit construction of a countably infinite non-abelian nilpotent group with modular topology lattice, and (2) establishing the absence of property $P_2$ in infinite abelian groups.
\end{abstract}

\maketitle
\section{Introduction}

For a non-empty set $X$, the collection of all (not necessarily Hausdorff) topologies on $X$ forms a complete lattice, ordered naturally by inclusion: $\tau \leq \sigma$ if and only if $\sigma$ is finer than $\tau$.
This lattice is sometimes referred to as the {\em toplattice} of $X$.
The study of toplattices was initiated by G. Birkhoff \cite{Birk}, who established some of their fundamental properties.
The subset $\mathcal{T}'_X$ of $T_1$ topologies in $\mathcal{T}_X$ is a sublattice, which is also complete.
In 1947, R. Vaidyanathaswamy \cite{Vaid} demonstrated that $\mathcal{T}_X$ fails to be distributive whenever $|X| \geq 3$.
Later, R. Bagley \cite{Bag} extended this result by proving in his 1954 PhD thesis that $\mathcal{T}'_X$ is not even modular.
However, a consequence of Ore’s theorem \cite[Theorem 13]{Ore} is that the lattice $\mathcal{T}_X$ satisfies the dual Birkhoff property, meaning that if $\tau$ and $\sigma$ are both covered by $\tau \vee \sigma$, then they both cover $\tau \wedge \sigma$.
In 1972, R. Larson and W. Thron established that $\mathcal{T}'_X$ satisfies both the Birkhoff property and its dual \cite{LT}.

For an abstract group $G$, the subset $\mathcal{L}_G$ of $\mathcal{T}_G$ consisting of all group topologies is of particular interest.
Equipped with the induced order from $\mathcal{T}_G$, the structure $\mathcal{L}_G$ forms a complete lattice, though it is merely a $\vee$-sub-semilattice of $\mathcal{T}_G$ and {\bf not} a sublattice.
The study of $\mathcal{L}_G$ has a rich history, with significant attention given to (non-)topologizable groups (i.e., groups admitting or failing to admit non-discrete Hausdorff group topologies) and minimal topological groups.
For a comprehensive survey on this subject, we refer the reader to \cite{Di} and \cite{HRT}.
In particular, Section 3 of \cite{HRT} selects fruitful results of Comfort concerning certain subsets of $\mathcal{L}_G$.
It is well known that $\mathcal{L}_G$ is modular when $G$ is abelian \cite{Lam}.
In this case, it has been established in \cite{Kil} that $\mathcal{L}_G$ contains precisely $2^{2^{|G|}}$ distinct Hausdorff group topologies, achieving the maximum possible number.
This result was extended by Dikranjan and Protasov \cite{Dik}, who proved that $G$ also has $2^{2^{|G|}}$ maximal group topologies (i.e., coatoms in $\mathcal{L}_G$ that are Hausdorff), which in turn implies that the width of $\mathcal{L}_G$ is $2^{2^{|G|}}$.
The first counterexample providing a group $G$ with non-modular $\mathcal{L}_G$ was obtained by Arnautov and Topal\u{a} \cite{ArTop}, who proved that a free group with at least two generators has non-modular lattice of group topologies.

In Section 14 of the renowned Kourovka Notebook, Arnautov formulated six questions concerning the lattice $\mathcal{L}_G$; see \cite[14.4 (a),(b) and 14.5 (a)-(d)]{KN}.
The problem 14.4 focuses on the modularity of $\mathcal{L}_G$ for nilpotent groups.
At the time of posing these questions, the situation was entirely unclear: no infinite non-abelian nilpotent group $G$ was known to have a modular or non-modular $\mathcal{L}_G$.
Thus, Arnautov asked \cite[14.4 (a) and (b)]{KN}:
\begin{itemize}
\item [(a)] {\it Is $\mathcal{L}_G$ modular for all nilpotent groups $G$? If not, then}
\item [(b)] {\it Does there exist a countably infinite nilpotent non-abelian group $G$ such that $\mathcal{L}_G$ is modular?}
\end{itemize}
Shortly after the problem was posed, Arnautov, together with Topal\u{a}, resolved (a) by providing a counterexample \cite{ATop}\footnote{While direct consultation of the source was unattainable, \cite{Ar2} substantiates that the proposed construction is rigorously demonstrated in the referenced work.}.
The first goal of this paper is to construct an example showing that the answer to (b) is affirmative.

\begin{theorem}\label{Th0}
There exists a countably infinite non-abelian nilpotent group $G$ such that $\mathcal{L}_G$ is modular.
\end{theorem}

For chains $C$ and $D$ in a poset $(P, \leq)$, we say $C$ \emph{refines} $D$ (or $C$ is a \emph{refinement} of $D$) if $D \subseteq C$ and every $c \in C$ satisfies $d_1 \leq c \leq d_2$ for some $d_1,d_2 \in D$. A chain admitting no proper refinements is called \emph{non-refinable}. Given $a \leq b$ in $P$, the interval $[a, b]$ consists of all elements $c$ with $a \leq c \leq b$, where maximal chains correspond precisely to non-refinable chains between these endpoints. While Arnautov \cite{Ar2} investigated finite non-refinable chains (composed of covering relations), extended chain analyses appear in \cite{CD1, CD2, CD3}.

For nilpotent groups $G$, although $\mathcal{L}_G$ need not be modular, Arnautov established fundamental structure:

\begin{theorem}\label{Th00}\cite[Theorem 4.6]{Ar2}
For any nilpotent group $G$:
\begin{itemize}
\item[(i)] $\mathcal{L}_G$ is semi-modular;
\item[(ii)] All finite non-refinable chains between fixed endpoints in $\mathcal{L}_G$ have equal length.
\end{itemize}
\end{theorem}

The Jordan-Hölder Chain Condition (Theorem \ref{JH}) guarantees conclusion (ii) as a consequence of semi-modularity in (i). Our second principal contribution provides an alternative proof methodology for this result.

For a positive integer $k$, a Hausdorff group topology $\tau$ on $G$ is called \emph{$k$-maximal} if every non-refinable chain from $\tau$ to the discrete topology $\tau_d$ has length exactly $k$. The special case of $1$-maximal topologies corresponds to coatoms in $\mathcal{L}_G$, traditionally called \emph{maximal group topologies}, which play pivotal roles in both minimal group theory and gap analysis within $\mathcal{L}_G$. The collection $\mathcal{A}_k(G)$ of all $k$-maximal topologies trivially forms an antichain, with $\mathcal{A}_0(G) := \{\tau_d\}$ by convention.

A group $G$ is said to satisfy \emph{property $P_k$} if every infinite non-refinable chain of Hausdorff topologies in $\mathcal{L}_G$ intersects $\mathcal{A}_k(G)$. While $P_1$ holds universally by Zorn's lemma, Arnautov's Kourovka Notebook question \cite[14.5(d)]{KN} challenges whether infinite abelian groups satisfy $P_k$ for all $k\in\mathbb{N}$. We demonstrate the failure of this conjecture at the first nontrivial level:

\begin{theorem}\label{Th1}
No infinite abelian group satisfies property $P_2$. Consequently, $P_k$ fails for all $k\geq2$.
\end{theorem}

\subsection*{Notations, Terminologies, and Elementary Facts}

Group topologies on a group $G$  are not assumed to be Hausdorff in general; consequently, $\mathcal{L}_G$ forms a complete lattice. The identity element of a group is denoted by $e$ or $1$. 

For any group $G$, the following result is known as {\em Pontryagin's Principle}, which characterizes when a system of subsets of $G$ forms a neighbourhood basis at $e$ for some group topology.

\begin{theorem}\label{Pont}\cite[Theorem 3.12]{AT}
Let $G$ be a group and $\mathcal{B}$ a family of subsets of $G$, each containing the identity $e$. Then $\mathcal{B}$ is a neighbourhood basis at $e$ for some group topology if and only if the following conditions hold:
\begin{itemize}
  \item[(i)] For each $U\in \mathcal{B}$, there exists $V\in \mathcal{B}$ such that $VV\subseteq U$;
  \item[(ii)] For each $U\in \mathcal{B}$, there exists $V\in \mathcal{B}$ such that $V^{-1}\subseteq U$;
  \item[(iii)] For each $U\in \mathcal{B}$ and $x\in U$, there exists $V\in \mathcal{B}$ such that $xV\subseteq U$;
  \item[(iv)] For each $U\in \mathcal{B}$ and $x\in G$, there exists $V\in \mathcal{B}$ such that $xVx^{-1}\subseteq U$;
  \item[(v)] For each pair $U, V\in\mathcal{B}$, there exists $W\in \mathcal{B}$ such that $W\subseteq U\cap V$.
\end{itemize}
Moreover, this topology is Hausdorff if and only if
\begin{itemize}
  \item[(vi)] $\bigcap \mathcal{B}=\{e\}$
\end{itemize}
is also satisfied.
\end{theorem}

If $\tau$ is a group topology on a group $G$ and $N$ is a subgroup of $G$, then $\tau\res_N$ denotes the subspace topology on $N$ induced from $(G, \tau)$.  
The quotient topology on the left coset space $G/N$ is denoted by $\tau/N$.  
It is evident that if $N$ is a normal subgroup, then $\tau/N$ is a group topology.  
The following lemma, known as Merzon's lemma, is widely used in the study of minimal groups.

\begin{lemma}\label{Le2.1}\cite{DS,Mer}
Let $G$ be a group and $N$ a (not necessarily normal) subgroup of $G$.  
Then for $\tau, \sigma\in \mathcal{L}_G$, if
\begin{itemize}
  \item[(i)] $\sigma\subseteq \tau$;
  \item[(ii)] $\sigma\res_N=\tau\res_N$; and
  \item[(iii)] $\sigma/N=\tau/N$,
\end{itemize}
then $\sigma=\tau$.
\end{lemma}

In a lattice $\mathcal{L}$, we say that {\em $a$ covers $b$}, or that {\em $a$ is a cover of $b$}, if $b < a$ and there is no $c\in \mathcal{L}$ with $b < c < a$.  
We also refer to the pair $\{b, a\}$ as a {\em covering (in $\mathcal{L}$)}.  
In this case, we write $b\prec a$.  
The symbol $b\preceq a$ means that either $b\prec a$ or $b=a$.  
Clearly, a covering is an non-refinable chain of length 1.  
Coverings in the lattice $\mathcal{L}_G$ are also called
{\em gaps}
\cite{HPTX1, HPTX2}.  
The supremum and infimum of two elements $a$ and $b$ in a lattice $\mathcal{L}$ are denoted by $a\vee b$ and $a\wedge b$, respectively.  
For a subset $A$ of $\mathcal{L}$, the supremum (resp. infimum) of $A$ (if it exists) is denoted by $\bigvee_{a\in A}a$ or $\bigvee A$ (resp. $\bigwedge_{a\in A}a$ or $\bigwedge A$).

The following fact is well known (see \cite{DPS}).

\begin{fact}\label{Fact}
Let $\sigma, \tau$ be two group topologies on a group $G$, and let $\mathcal{B}_\sigma$ (resp., $\mathcal{B}_\tau$) be a neighbourhood basis at the identity in $(G, \sigma)$ (resp., $(G, \tau)$). Then:
\begin{itemize}
  \item[(i)] $\{U\cap V: U\in \mathcal{B}_\sigma,\, V\in\mathcal{B}_\tau\}$ forms a local basis at $e$ in $(G, \sigma\vee\tau)$;
  \item[(ii)] If the centre of $G$ is open with respect to either $\sigma$ or $\tau$, then $\{UV: U\in \mathcal{B}_\sigma,\, V\in\mathcal{B}_\tau\}$ forms a local basis at $e$ in $(G, \sigma\wedge\tau)$.
\end{itemize}
\end{fact}

\begin{proposition}\label{Prop:Oct11}
Let $G$ be a group and $N$ a subgroup. Then for every pair $\sigma, \tau\in \mathcal{L}_G$, one has
 \begin{itemize}
   \item[(i)] $(\sigma\vee\tau)\res_N = \sigma\res_N \vee \tau\res_N$;
   \item[(ii)] $(\sigma\wedge\tau)/N = \sigma/N \wedge \tau/N$, provided that $N$ is normal.
 \end{itemize}
\end{proposition}
\begin{proof}
Item (i) is an immediate consequence of Fact \ref{Fact} (i).

For (ii), note first that $(\sigma\wedge\tau)/N \leq \sigma/N \wedge \tau/N$.  
Conversely, let $\lambda$ be the group topology on $G$ defined such that a subset is open if and only if it is the preimage of some open set under the canonical map 
$$
G\to (G/N, \sigma/N \wedge \tau/N).
$$
Then $\lambda\leq \sigma\wedge\tau$, which implies $(\sigma\wedge\tau)/N\geq \lambda/N = \sigma/N \wedge \tau/N$.
\end{proof}

A lattice $\mathcal{L}$ is called {\em modular} if it satisfies the following condition: for every triple $a, b, c\in \mathcal{L}$ with $a\leq c$, one has 
$$
a\vee(b\wedge c) = (a\vee b)\wedge c.
$$
Distributive lattices and the lattice of normal subgroups of a given group $G$ are classic examples of modular lattices.  
Modular lattices possess several important properties; among them is the so-called {\em upper cover condition}: if $b\prec a$, then for every $c$ one has $b\vee c\preceq a\vee c$.

A lattice $\mathcal{L}$ is called {\em semi-modular} or {\em upper semi-modular} if it satisfies the upper covering condition.  
A key property of semi-modular lattices is that they satisfy the {\em Jordan-H\"{o}lder Chain Condition} (also known as the {\em Jordan-Dedekind Chain Condition}).

\begin{definition}\cite[$\S 1.9$]{St}
A lattice $\mathcal{L}$ is said to have the Jordan-H\"{o}lder Chain Condition if for every pair of elements $a, b\in \mathcal{L}$ with $a<b$, either (1) all maximal chains in $[a, b]$ are infinite, or (2) all maximal chains in $[a,b]$ are finite and have the same length.
\end{definition}

\begin{theorem}\label{JH}\cite[Theorem 1.9.1]{St}
Every semi-modular lattice satisfies the Jordan-H\"{o}lder Chain Condition.
\end{theorem}

Undefined symbols or notions may be found in \cite{AT, DPS} for Topological Groups and in \cite{Gr} for Lattice Theory.

\section{(Semi-)Modularity of $\mathcal{L}_G$ with $G$ Nilpotent}

We now proceed to prove Theorem \ref{Th0}. First, we recall several foundational results. The following three lemmas are standard and their proofs follow directly from definitions, hence we omit them.

\begin{lemma}\label{Le00}
Let $G$ be an arbitrary group (finite or infinite). The collection of all normal subgroups of $G$, ordered by reverse inclusion, constitutes a modular lattice.
\end{lemma}

\begin{lemma}\label{Le01}
Let $G$ be a finite group equipped with a group topology $\tau$. There exists a normal subgroup $H$ of $G$ such that:
\begin{itemize}
\item The singleton $\{H\}$ forms a neighborhood base at the identity for $(G, \tau)$;
\item A subset of $G$ is open if and only if it is a union of cosets of $H$.
\end{itemize}
\end{lemma}

\begin{lemma}\label{Le02}
Let $\mathcal{P}_1$ and $\mathcal{P}_2$ be modular lattices. Their Cartesian product $\mathcal{P}_1 \times \mathcal{P}_2$, endowed with the componentwise ordering (where $(a,b) \leq (c,d)$ iff $a \leq c$ and $b \leq d$), remains modular.
\end{lemma}

The following lemma, while elementary, requires explicit verification:

\begin{lemma}\label{Le03}
Let $n \geq 2$ and $m \geq 2$ be coprime integers, and let $G$ be a group where every element has order dividing $n$. For any subgroup $H \leq G$ and element $x \in G$, if $x^m \in H$, then necessarily $x \in H$.
\end{lemma}

\begin{proof}
Since the order of $x$ divides $n$ and $m,n$ are coprime, the elements $x$ and $x^m$ share the same order. Consequently, $x$ belongs to the cyclic subgroup $\langle x^m \rangle$, which is contained in $H$.
\end{proof}

\begin{proof}[{\bf Proof of Theorem \ref{Th0}}]
Let $F$ be a finite non-abelian nilpotent group. Fix a prime number $p$ such that $n := |F|$ is coprime to $p$, and let $H = \bigoplus_\omega \mathbb{Z}(p)$, where $\mathbb{Z}(p)$ denotes the cyclic group of order $p$. We claim that $G := H \times F$ satisfies the required properties.

By Lemmas \ref{Le00}, \ref{Le01}, and \ref{Le02}, it suffices to show that every group topology $\tau$ on $G$ is a product topology. Specifically, there exist topologies $\tau_H \in \mathcal{L}_H$ and $\tau_F \in \mathcal{L}_F$ such that $\tau = \tau_H \times \tau_F$. Throughout the proof, we identify $H$ and $F$ with their canonical embedding images in $G$, and denote the identity element of $G$ by $e = (1,1)$.

Let $N$ denote the closure of $\{e\}$ in $(G, \tau)$. We first establish that $N = N_1 \times N_2$, where $N_1 = N \cap H$ and $N_2 = N \cap F$. Take $(x, y) \in N \subseteq H \times F$ with $x \in H$ and $y \in F$. Observe that:
\[
(x^n, 1) = (x, y)^n \in N \quad \text{and} \quad (1, y^p) = (x, y)^p \in N.
\]
Since $n$ and $p$ are coprime, Lemma \ref{Le03} implies $(x, 1) \in N_1$ and $(1, y) \in N_2$. Therefore, $N \subseteq N_1 \times N_2$, while the reverse inclusion is immediate. Hence $N = N_1 \times N_2$.

Consider the continuous homomorphism $\varphi: G \to G$ defined by $\varphi(g) = g^p$. From the previous decomposition, $\varphi(x,y) \in N$ if and only if $x^p \in N_1$ and $y^p \in N_2$. Noting that $x^p = 1$ for all $x \in H$ and $y^p \in N_2$ precisely when $y \in N_2$, we find:
\[
K := \varphi^{-1}(N) = H \times N_2.
\]
This closed subgroup has finite index in $G$, hence is open. Let $U$ be an open neighborhood of $e$ contained in $K$. Since $N$ is the intersection of all open neighborhoods of $e$, $N\subseteq U$.
Hence we have $N_2 \subseteq U \subseteq H \times N_2$. This implies:
\begin{equation}\label{decom}
U = U_1 \times N_2
\end{equation}
where $U_1$ is open in $(H, \tau_H)$ with $\tau_H$ being the subspace topology inherited from $(G, \tau)$. Let $\tau_F$ denote the group topology on $F$ with local base $\{N_2\}$. Then $U = U_1 \times N_2 \in \tau_H \times \tau_F$, establishing $\tau \leq \tau_H \times \tau_F$.

Conversely, for any open neighborhood $V$ of $e$ in $(H, \tau_H)$, choose an open neighborhood $U$ of $e$ in $(G, \tau)$ with $U \cap H \subseteq V$. Assuming $U \subseteq K$, decomposition \eqref{decom} gives $U = U_1 \times N_2$ where $U_1 \subseteq V$. Thus $U \subseteq V \times N_2$, proving $\tau_H \times \tau_F \leq \tau$. We conclude that $\tau = \tau_H \times \tau_F$, as required.
\end{proof}

\begin{remark}\label{RemModularExtension}
The nilpotency condition on $F$ is not essential to the argument presented. This observation permits the construction of groups with modular lattice structures on their group topologies beyond the nilpotent category. Specifically, one can systematically produce:
\begin{itemize}
\item Non-nilpotent finite groups
\item Infinite groups with torsion elements
\end{itemize}
that maintain the modularity property in their lattice of group topologies. The key mechanism resides in the product topology construction combined with coprime order arguments from Lemma \ref{Le03}.
\end{remark}
We now turn to the second aim of this section, say, to prove the lattice $\mathcal{L}_G$ is semi-modular when $G$ is nilpotent.
The following lemma comes from \cite[Proposition 3.1]{HPTX1}.

\begin{lemma}\label{Le2.2}
Let $G$ be an infinite group and $N$ a normal subgroup of $G$. Suppose $\sigma, \tau\in \mathcal{L}_G$ such that $\sigma\prec\tau$, then
\begin{itemize}
  \item [(a)] $\sigma/N\preceq \tau/N$ (in the lattice $\mathcal{L}_{G/N}$); and
  \item [(b)] if $N$ is additionally assumed to be central, then $\sigma\res_N\preceq \tau\res_N$ (in the lattice $\mathcal{L}_N$).
\end{itemize}
\end{lemma}

\begin{corollary}\label{Coro2.3}
Let $G$ be a group with a central subgroup $N$ and $\sigma\leq \tau\in \mathcal{L}_G$. Then $\sigma\prec\tau$ if and only if one of the following holds:
\begin{itemize}
  \item [(a)] $\sigma\res_N=\tau\res_N$ and $\sigma/N\prec \tau/N$;
  \item [(b)] $\sigma\res_N\prec \tau\res_N$ and $\sigma/N=\tau/N$.
\end{itemize}
\end{corollary}

\begin{proof}
Suppose first that $\sigma\prec \tau$. Then $\sigma\res_N\preceq \tau\res_N$ and $\sigma/N\preceq \tau/N$, according to Lemma \ref{Le2.2}. Since $\sigma\neq \tau$, by Lemma \ref{Le2.1} these two equalities cannot hold simultaneously. Thus the necessity is proved.

Conversely, let $\sigma\lneq \lambda\leq \tau$. If (a) holds, then $\sigma\res_N=\lambda\res_N=\tau\res_N$, and by Lemma \ref{Le2.1} we deduce $\lambda/N\neq \sigma/N$. Since $\sigma/N\prec \tau/N$, it follows that $\lambda/N=\tau/N$. Applying Lemma \ref{Le2.1} again, we obtain $\lambda=\tau$. The case (b) follows by a similar argument.
\end{proof}

\begin{remark}\label{Re1}{\em It is an easy exercise that one may replace all “$\prec$” by “$\preceq$” in the above corollary.}\end{remark}

The next lemma is a generalization of Fact \ref{Fact} (ii).

\begin{lemma}\label{Le:June1}
Let $G$ be a group with $N$ a central subgroup. If $\tau, \sigma\in \mathcal{L}_G$ satisfy $\sigma/N\leq \tau/N$, then the family
$$
\mathcal{B}=\{UV: U\in \mathcal{B}_\tau,\; V\in\mathcal{B}_\sigma\}
$$
forms a local basis at $e$ in $(G, \sigma\wedge\tau)$, where $\mathcal{B}_\tau$ (resp., $\mathcal{B}_\sigma$) is a neighbourhood basis of $e$ in $(G, \tau)$ (resp., $(G, \sigma)$).
\end{lemma}

\begin{proof}
We must show that the family $\mathcal{B}$ satisfies conditions (i)–(v) in Theorem \ref{Pont}. Take arbitrary $U\in \mathcal{B}_\tau$ and $V\in \mathcal{B}_\sigma$.

First, we prove (ii). Choose a symmetric neighbourhood $V_0$ of $e$ in $(G, \sigma)$ such that
$$
V_0V_0V_0\subseteq V.
$$
By the inequality $\sigma/N\leq \tau/N$, one can also find a symmetric neighbourhood $U_0$ of $e$ in $(G, \tau)$ satisfying
\begin{equation}\label{eq1}
U_0\subseteq V_0N.
\end{equation}
Shrinking $U_0$ if necessary, assume that $U_0\subseteq U$. Now take an arbitrary $x\in (U_0V_0)^{-1}=V_0U_0$. Write $x=vu$ with $u\in U_0$ and $v\in V_0$. By (\ref{eq1}), we can further choose $v'\in V$ and $n\in N$ such that
$$
u=v'n=nv'.
$$
Since $n$ is central in $G$, we have
$$
x=vu=vv'n=nvv'=nv'v'^{-1}vv'=uv'^{-1}vv'\in U_0V_0V_0V_0\subseteq UV.
$$
Thus, $(U_0V_0)^{-1}\subseteq UV$, verifying (ii).

Next, we verify (i). Take a neighbourhood $U_1\in \mathcal{B}_\tau$ (resp., $V_1\in \mathcal{B}_\sigma$) such that $U_1U_1\subseteq U$ (resp., $V_1V_1\subseteq V$). By the property (ii) already proved, choose symmetric identity neighbourhoods $U_2$ in $(G, \tau)$ and $V_2$ in $(G, \sigma)$ such that
$$
V_2U_2\subseteq U_1V_1.
$$
Assume $U_2\subseteq U_1$ and $V_2\subseteq V_1$. Then
$$
(U_2V_2)(U_2V_2)=U_2(V_2U_2)V_2\subseteq U_1(U_1V_1)V_1=(U_1U_1)(V_1V_1)\subseteq UV.
$$

For (iii), let $x\in UV$, so that $x=uv$ with $u\in U$ and $v\in V$. Choose $U_0\in \mathcal{B}_\tau$ such that 
$$
U_0u\subseteq U,
$$
and $V_0\in \mathcal{B}_\sigma$ such that
$$
(u^{-1}V_0u)v\subseteq V.
$$
Then,
$$
U_0V_0x=U_0V_0uv=U_0(uu^{-1})V_0uv=(U_0u)(u^{-1}V_0u)v\subseteq UV.
$$

For (iv), for any $x\in G$, choose $U'\in \mathcal{B}_\tau$ and $V'\in \mathcal{B}_\sigma$ with
$$
xU'x^{-1}\subseteq U \quad\text{and}\quad xV'x^{-1}\subseteq V.
$$
Then,
$$
xUVx^{-1}=xUx^{-1}\,xVx^{-1}\subseteq UV.
$$

Finally, (v) is straightforward. For any $U_1, U_2\in \mathcal{B}_\tau$ and $V_1, V_2\in \mathcal{B}_\sigma$, choose $U_3\in \mathcal{B}_\tau$ and $V_3\in \mathcal{B}_\sigma$ with
$$
U_3\subseteq U_1\cap U_2 \quad\text{and}\quad V_3\subseteq V_1\cap V_2.
$$
Then,
$$
U_3V_3\subseteq (U_1\cap U_2)(V_1\cap V_2)\subseteq U_1U_2\cap V_1V_2.
$$
This completes the verification.
\end{proof}

\begin{remark}
It is easy to see that the family
$$
\{VU: U\in \mathcal{B}_\tau,\; V\in\mathcal{B}_\sigma\}
$$
is also a neighbourhood basis at $e$ in $(G, \sigma\wedge\tau)$.
\end{remark}

\begin{lemma}\label{Le:Oct1}
For a group $G$ with a normal subgroup $N$ and a pair $\sigma, \tau$ of group topologies on $G$, let $\tau^*$ be the group topology on $G$ with the identity neighbourhood basis $\{UN: e\in U\in \tau\}$. Then
$$
(\sigma\vee\tau^*)/N=\sigma/N\vee\tau/N.
$$
\end{lemma}

\begin{proof}
Let $\pi: G\to G/N$ be the canonical mapping. Since
$$
\{V\cap UN: e\in V\in \sigma,\; e\in U\in \tau\}
$$
forms an identity neighbourhood basis for $(G, \sigma\vee\tau^*)$, the sets
$$
\pi(V\cap UN)=\pi(V)\cap \pi(U)
$$
form an identity neighbourhood basis for $(G/N, (\sigma\vee\tau^*)/N)$. On the other hand, the families $\{\pi(V): e\in V\in \sigma\}$ and $\{\pi(U): e\in U\in \tau\}$ are identity neighbourhood bases for $(G/N, \sigma/N)$ and $(G/N, \tau/N)$ respectively. Hence, by Proposition \ref{Prop:Oct11}, the equality
$$
(\sigma\vee\tau^*)/N=\sigma/N\vee\tau/N
$$
has been obtained.
\end{proof}

Now we are going to give the main result of this section.

\begin{theorem}\label{SemiMod}
Let $G$ be a group and $N$ a central subgroup. Then $\mathcal{L}_G$ is semi-modular if and only if $\mathcal{L}_{G/N}$ is semi-modular.
\end{theorem}

\begin{proof}
We first prove the sufficiency. Assume that $\mathcal{L}_{G/N}$ is semi-modular. Let $\tau_1, \tau_2\in \mathcal{L}_G$ with $\tau_1\prec \tau_2$. We must show that for every $\sigma\in \mathcal{L}_G$, 
$$
\tau_1\vee\sigma\preceq\tau_2\vee\sigma.
$$
Note that $\tau_i\vee \sigma=\tau_i\vee (\tau_1\vee \sigma)$ for $i=1,2$, so we may assume $\tau_1\leq \sigma$. Our aim is then to prove that $\sigma\preceq \tau_2\vee\sigma$.

Let $\mathcal{B}_1$, $\mathcal{B}_2$, and $\mathcal{B}_\sigma$ be neighbourhood bases at $e$ in $(G, \tau_1)$, $(G, \tau_2)$, and $(G, \sigma)$ respectively. Define $\sigma^*$ to be the group topology on $G$ with the local basis
$$
\{UN: U\in \mathcal{B}_\sigma\},
$$
and set
$$
\tau_i'=\tau_i\vee\sigma^*, \quad i=1,2.
$$
Then $\sigma^*/N=\sigma/N$. Since $\sigma^*\leq \sigma$, we have
\begin{equation}\label{e1}
\tau_i'\vee \sigma=(\tau_i\vee\sigma^*)\vee \sigma=\tau_i\vee\sigma,\quad (i=1,2),
\end{equation}
and $\tau_1'\leq \sigma$. By Lemma \ref{Le:Oct1},
\begin{equation}\label{e2}
\tau_i'/N=\tau_i/N\vee\sigma/N,\quad (i=1,2).
\end{equation}
Since $\tau_1\leq \sigma$, it follows that
\begin{equation}\label{e3}
\tau_1'/N=\tau_1/N\vee\sigma/N=\sigma/N.
\end{equation}

\noindent{\bf Claim.} $\tau_1'\preceq\tau_2'$.

\begin{proof}[Proof of Claim]
Apply Corollary \ref{Coro2.3} (and Remark \ref{Re1}). If (a) holds for $\tau_1$ and $\tau_2$, then $\tau_1'/N\preceq \tau_2'/N$ by the semi-modularity of $\mathcal{L}_{G/N}$ and (\ref{e2}). Moreover, it is easy to check that
$$
\tau_1'\res_N=\tau_1\res_N=\tau_2\res_N=\tau_2'\res_N.
$$
Thus, by Remark \ref{Re1}, $\tau_1'\preceq \tau_2'$. On the other hand, if (b) holds, then
$$
\tau_1'/N=\tau_1/N\vee\sigma/N=\tau_2/N\vee\sigma/N=\tau_2'/N.
$$
Since $N$ is central, the lattice $\mathcal{L}_N$ is modular, so
$$
\tau_1'\res_N=\tau_1\res_N\prec \tau_2\res_N=\tau_2'\res_N,
$$
and by Corollary \ref{Coro2.3}, $\tau_1'\prec \tau_2'$.
\end{proof}

By (\ref{e1}), if $\tau_1'=\tau_2'$, then $\tau_1\vee\sigma=\tau_2\vee\sigma$ and we are done. Hence, assume $\tau_1'\prec \tau_2'$. Let $\mathcal{B}'_1$ and $\mathcal{B}'_2$ be local neighbourhood bases in $(G, \tau_1')$ and $(G, \tau_2')$ respectively. We now consider two cases according to Corollary \ref{Coro2.3}:

\medskip

\noindent{\bf Case 1:} $\tau_1'\res_N=\tau_2'\res_N$ and $\tau_1'/N\prec\tau_2'/N$.

Denote by $\tau_2'^*$ the group topology on $G$ with local basis 
$$
\{VN:V\in \mathcal{B}'_2\}
$$
at $e$. Then $\tau_1'\leq \tau_1'\vee \tau_2'^*\leq \tau_2'$. By Lemma \ref{Le:Oct1},
$$
(\tau_1'\vee\tau_2'^*)/N=\tau_1'/N\vee \tau_2'/N=\tau_2'/N.
$$
Also,
$$
(\tau_1'\vee\tau_2'^*)\res_N=\tau_1'\res_N=\tau_2'\res_N.
$$
Then by Lemma \ref{Le2.1}, $\tau_1'\vee\tau_2'^*=\tau_2'$. This, together with $\tau_1'\leq \sigma$, implies
$$
\tau_2'\vee \sigma=\tau_1'\vee\tau_2'^*\vee\sigma=\tau_2'^*\vee\sigma.
$$
Thus, by Lemma \ref{Le:Oct1},
$$
(\tau_2'\vee\sigma)/N=(\tau_2'^*\vee\sigma)/N=\tau_2'/N\vee\sigma/N.
$$
From (\ref{e3}),
$$
\tau_2'/N\vee \sigma/N=\tau_2'/N\vee\tau_1'/N=\tau_2'/N\succ\tau_1'/N=\sigma/N.
$$
Hence, $\sigma/N\prec \tau_2'/N$. On the other hand,
$$
(\tau_2'\vee \sigma)\res_N=\tau_2'\res_N\vee\sigma\res_N=\tau_1'\res_N\vee\sigma\res_N=\sigma\res_N.
$$
Applying Corollary \ref{Coro2.3} and (\ref{e1}), we obtain
$$
\sigma\prec \tau_2'\vee\sigma=\tau_2\vee\sigma.
$$

\medskip

\noindent{\bf Case 2:} $\tau_1'\res_N\prec\tau_2'\res_N$ and $\tau_1'/N=\tau_2'/N$.

Without loss of generality, assume $\tau_2'\nleq \sigma$ (otherwise, $\sigma=\sigma\vee\tau_2'$ and there is nothing to prove). Take $\lambda\in \mathcal{L}_G$ with $\sigma\leq \lambda< \sigma\vee\tau_2'$. We must show that $\lambda= \sigma$. Let $\mathcal{B}_\lambda$ be a neighbourhood basis at $e$ in $(G, \lambda)$. For any $V\in \mathcal{B}_\lambda$, choose $V_1\in \mathcal{B}_\lambda$ such that $V_1V_1\subseteq V$. Since $\lambda<\sigma\vee\tau_2'$, choose $U_1\in \mathcal{B}_\sigma$ and $W\in \mathcal{B}'_2$ with
$$
U_1\cap W\subseteq V_1.
$$
Because $\sigma\leq \lambda$, there exist $V_2\in \mathcal{B}_\lambda$ and $U_2\in \mathcal{B}_\sigma$ satisfying
$$
U_2V_2^{-1}\subseteq U_1.
$$
Assume $V_2\subseteq V_1$. Note that $\tau_2'/N=\tau_1'/N=\sigma/N\leq \lambda/N$. By Lemma \ref{Le:June1}, the family 
$$
\{WV: W\in \mathcal{B}'_2,\; V\in \mathcal{B}_\lambda\}
$$
forms an identity basis for $\lambda\wedge \tau_2'$. Since $\lambda$ is strictly less than $\sigma\vee\tau_2'$, we have $\tau_2'\nleq \lambda$. Thus, $\tau_1'\leq \lambda\wedge\tau_2'<\tau_2'$, which implies $\lambda\wedge\tau_2'=\tau_1'$. In particular, $\lambda\wedge\tau_2'\leq \sigma$. Therefore, there exists $U_3\in \mathcal{B}_\sigma$ such that
$$
U_3\subseteq WV_2.
$$
Then for any $u\in U_3$, there exist $w\in W$ and $v\in V_2$ with $u=wv$. Hence,
$$
w=uv^{-1}\in U_2V_2^{-1}\subseteq U_1.
$$
Thus,
$$
w\in U_1\cap W\subseteq V_1.
$$
Finally,
$$
u=wv\in V_1V_2\subseteq V_1V_1\subseteq V.
$$
It follows that $U_3\subseteq V$, so $\lambda\leq \sigma$. Therefore, $\lambda=\sigma$, proving $\sigma\prec \sigma\vee\tau_2'$.
 
In either case, we deduce
$$
\sigma\preceq \sigma\vee\tau_2'=\sigma\vee\tau_2, \quad \text{(by (\ref{e1}))}.
$$
Hence, $\mathcal{L}_G$ is semi-modular.

Now suppose that $\mathcal{L}_G$ is semi-modular. Let $\tau_0$ be the anti-discrete topology on $G$, i.e., the least element in $\mathcal{L}_G$, and let $\tau_N$ be the group topology on $G$ with the identity neighbourhood basis $\{N\}$. Evidently, the mapping
$$
\pi: [\tau_0, \tau_N]\to \mathcal{L}_{G/N}
$$
sending $\tau$ to $\tau/N$ is an isomorphism of lattices. Since semi-modularity is preserved by taking intervals, $\mathcal{L}_{G/N}$ is semi-modular.
\end{proof}

\begin{corollary}[Theorem \ref{Th00}]
Let $G$ be a nilpotent group. Then $\mathcal{L}_G$ is semi-modular.
\end{corollary}

\begin{proof}
We prove this by induction on the nilpotency class of $G$. First, if $G$ is abelian, then $\mathcal{L}_G$ is modular. Now, assume that $G$ is nilpotent of class $n$ and that $\mathcal{L}_H$ is semi-modular for every nilpotent group $H$ of class $\leq n-1$. Let $Z$ be the centre of $G$. Then $G/Z$ is nilpotent of class $n-1$, so $\mathcal{L}_{G/Z}$ is semi-modular. Now apply Theorem \ref{SemiMod}.
\end{proof}

One may naturally inquire about the behavior of infinite chains in (ii) of Theorem \ref{Th00}. More precisely, for a nilpotent group $G$, must two (potentially infinite) non-refinable chains in $\mathcal{L}_G$ with identical endpoints necessarily share the same cardinality? We recall that the answer is affirmative when at least one chain is finite. However, this property fails in general, even for abelian groups. We establish this through the following construction, adopting additive notation for abelian groups with identity element $0$.

\begin{lemma}\label{Le:Densub}
Let $G = \mathbb{Z}(p)^\omega$ where $p$ is prime. Then $G$ contains two non-refinable chains of dense subgroups sharing the same endpoints, where one chain has cardinality $\omega$ and the other $2^\omega$.
\end{lemma}

\begin{proof}
Let $H,K$ be subgroups of $G$ with $H \cap K = \{0\}$, where $H$ is dense and $K$ is countably infinite. Considering $K$ as a vector space over $\mathbb{Z}(p)$, let $B = \{x_n : n \in \omega\}$ be a basis. For each $n \in \omega$, define $K_n$ as the subspace generated by $\{x_1, \ldots, x_n\}$. Setting $P_n = H + K_n$, we obtain the chain:
\[
\mathcal{C} := \{H, H + K_0, H + K_1, \ldots, H + K_n, \ldots, H + K\}
\]
which constitutes a non-refinable chain of dense subgroups.

By the Erd\H{o}s–Tarski theorem \cite[p.54]{CD1}, the power set of $B$ contains a maximal chain $\mathcal{D}$ of cardinality $2^\omega$. For each $D \in \mathcal{D}$, let $K_D$ be the subspace generated by $D$. Then:
\[
\mathcal{C}' := \{H + K_D : D \in \mathcal{D}\}
\]
forms another non-refinable chain with endpoints $H$ and $H + K$, inheriting the cardinality $2^\omega$ from $\mathcal{D}$.
\end{proof}

Recall that a topological group $G$ is \emph{precompact} if for every neighborhood $U$ of the identity, there exists a finite subset $F \subseteq G$ such that $G = FU$. A key result states that every Hausdorff precompact group topology densely embeds into a compact group. The collection $\mathcal{PK}_G$ of precompact group topologies forms a sublattice of $\mathcal{L}_G$. 

In their foundational work \cite{CR}, Comfort and Ross established that for abelian groups, every Hausdorff precompact topology $\tau$ is generated by a subgroup of continuous homomorphisms from $G$ to the 1-dimensional torus $\mathbb{T}$. These homomorphisms correspond precisely to dense subgroups of $\widehat{G}$ (the Pontryagin dual group of $G$). 

Crucially, this density requirement fails for non-Hausdorff precompact topologies, where the associated subgroup of $\widehat{G}$ may be arbitrary. This yields a lattice isomorphism between:
\begin{itemize}
    \item The sublattice $\mathcal{PK}_G$ of precompact topologies
    \item The lattice of subgroups of $\widehat{G}$ (ordered by inclusion)
\end{itemize}
This correspondence was first Proved by Remus in \cite{Rem}.
\begin{theorem}[Comfort-Ross Duality]
Let $G$ be a discrete abelian group with Pontryagin dual $\widehat{G}$. The mapping
\[
\varphi: \mathcal{S}(\widehat{G}) \to \mathcal{PK}_G
\]
defined by letting $\varphi(H)$ be the initial topology induced by $H \subseteq \widehat{G}$, constitutes a lattice isomorphism. Moreover, $\varphi(H)$ is Hausdorff if and only if $H$ is dense in $\widehat{G}$.
\end{theorem}

This duality yields the following counterpart to Lemma \ref{Le:Densub}:

\begin{corollary}
Let $H = \bigoplus_\omega \mathbb{Z}(p)$ be the direct sum of countably many copies of $\mathbb{Z}(p)$. Then $\mathcal{L}_H$ contains two non-refinable chains of Hausdorff precompact topologies with identical endpoints, one of cardinality $\omega$ and the other of cardinality $2^\omega$.
\end{corollary}

Regarding lattice-theoretic properties, recall that a \emph{dually semi-modular lattice} (or \emph{lower semi-modular lattice}) satisfies the lower covering condition: if $b \prec a$, then $b \wedge c \preceq a \wedge c$ for all $c$. Such lattices inherently satisfy the Jordan-Hölder chain condition due to dual invariance. While one might conjecture that Theorem \ref{Th00} extends to dual semi-modularity, Arnautov's 2012 counterexample \cite{Ar12} demonstrates this fails in general.

\section{Proof of Theorem \ref{Th1}}

We introduce a new property for topological groups that exhibits deep connections with the $P_k$ properties.
Throughout this section, abelian groups are written additively. So we denote by $0$ the identity element.

\begin{definition}
A Hausdorff group topology $\tau$ on a group $G$ is said to satisfy \emph{Property Q} if every chain in the partially ordered set $\{\tau' \in \mathcal{L}_G : \tau' < \tau\}$ of Hausdorff group topologies admits an upper bound.
\end{definition}

We first establish a foundational lemma:

\begin{lemma}\label{le0}
Let $k \in \N$ and $G$ be an abelian group. Then $G$ satisfies $P_{k+1}$ if and only if:
\begin{enumerate}
    \item $G$ satisfies $P_k$;
    \item Every topology $\tau \in \mathcal{A}_k(G)$ possesses Property Q.
\end{enumerate}
\end{lemma}

\begin{proof}
The forward implication ($P_{k+1} \Rightarrow P_k$) is immediate. Now assume $G$ satisfies $P_{k+1}$, $\tau\in \mathcal{A}_k(G)$ and $\mathcal{C}$ is a maximal chain of Hausdorff group topologies in $\{\tau'\in \mathcal{L}_G: \tau'<\tau\}$.
 Denote by $\mathcal{C}^*$ an non-refinable chain in $\mathcal{L}_G$ containing $\mathcal{C}\cup\{\tau\}$ and ending at the discrete topology.
 Then $P_{k+1}$ property ensures us to find a $(k+1)$-maximal topology $\tau^*$ in $\mathcal{C}^*$.
So $\tau^*<\tau$ and $\tau^*$ is an upper bound of $\mathcal{C}$ in the set $\{\tau'\in \mathcal{L}_G: \tau'<\tau\}$.

For the converse, assume $G$ satisfies $P_k$ and let $\mathcal{C}$ be a non-refinable chain of Hausdorff group topologies terminating at the discrete topology. By $P_k$, select $\tau \in \mathcal{A}_k(G) \cap \mathcal{C}$. Property Q ensures that $\{\tau' < \tau : \tau' \in \mathcal{C}\}$ contains a maximal element $\tau^*$. The non-refinability of $\mathcal{C}$ then implies $\tau^* \in \mathcal{A}_{k+1}(G)$.
\end{proof}

\begin{lemma}\label{le1}
Let $G$ be a topological abelian group endowed with a maximal group topology $\tau$. Then $(G, \tau)$ is not first-countable.
\end{lemma}

\begin{proof}
It suffices to demonstrate the absence of non-trivial convergent sequences. Since maximal group topologies on abelian groups refine the Bohr topology \cite{DPS}, the conclusion follows directly from Leptin's theorem \cite[Theorem 9.9.30]{AT}, which asserts that compact subsets in the Bohr topology are necessarily finite.
\end{proof}

\begin{lemma}\label{le2}
Let $\tau$ be a Hausdorff group topology with Property Q on an infinite abelian group $G$. Then $\chi(G, \tau) \leq |G|$. In particular, $(G, \tau)$ is first-countable when $G$ is countable.
\end{lemma}

\begin{proof}
Suppose $\kappa := \chi(G, \tau) > \lambda := |G|$. By Arhangel'ski\u{\i}'s theorem \cite{Arh}, there exists a Hausdorff topology $\sigma < \tau$ with $\chi(G, \sigma) \leq \lambda$. Let $\{U_\alpha : \alpha < \kappa\}$ be a neighborhood base at $0$. For each $\alpha$, construct a continuous homomorphism $f_\alpha: (G, \tau) \to H_\alpha$ onto a metrizable group such that $U_\alpha$ contains a preimage of the unit ball of $0$ in $H_\alpha$ \cite[Theorem 3.4.18]{AT}, so inducing topology $\sigma_\alpha<\tau$ in which $U_\alpha$ is a neighbourhood of $0$. Define transfinite topologies:

\[
\tau_\beta = \left(\bigvee_{\alpha \leq \beta} \sigma_\alpha\right) \vee \sigma \quad \text{for } \beta < \kappa.
\]

Each $\tau_\beta$ satisfies $\chi(G, \tau_\beta) \leq \lambda \cdot |\beta| < \kappa$, hence $\tau_\beta < \tau$. As $\tau = \bigvee_{\alpha<\kappa} \tau_\alpha$, the chain $\{\tau_\alpha\}$ lacks an upper bound in $\{\tau' \in \mathcal{L}_G : \tau' < \tau\}$, contradicting Property Q.
\end{proof}

\begin{proof}[Proof of Theorem \ref{Th1}]
For countable groups $G$, the result follows directly from Lemmas \ref{le0}-\ref{le2}. For uncountable $H$, consider a countable subgroup $G$. Any topology $\lambda$ on $G$ extends uniquely to a topology $\lambda^*$ on $H$ making $G$ open. Let $\Omega \subset \mathcal{L}_G$ be a non-refinable chain avoiding $\mathcal{A}_2(G)$. Then $\Omega^* := \{\lambda^* : \lambda \in \Omega\}$ inherits non-refinability. The existence of a 2-maximal topology $\tau^* \in \Omega^*$ would imply $\tau^*\res_G \in \Omega \cap \mathcal{A}_2(G)$, yielding a contradiction.
\end{proof}

\begin{remark} There is a stronger version of Lemma \ref{le1}. The case $G=\Z$ was proved in \cite{HPTX1}; that proof is also suitable for the following generalized result. For readers' convenience, we provide the proof here.

Let us recall that a family $\mathcal{F}$ of subsets of a countably infinite set $X$ is called to have {\em strongly finite intersection property} if the intersection of any finite members of $\mathcal{F}$ is an infinite set.
A subset $A$ of $X$ is called a {\em pseudo-intersection} of $\mathcal{F}$ if $A\setminus F$ is finite for every $F\in \mathcal{F}$.
We define $\mathfrak{p}$ as the minimal cardinality of a family of subsets of $X$ with strongly finite intersection property but without infinite pseudo-intersection.
It is known that $\omega_1\leq \mathfrak{p}\leq \mathfrak{c}$ (see \cite[Theorem 3.1]{vanD}).

Now suppose that $(G, \tau)$ is a countable abelian group with a maximal group topology.
If $\chi(G, \tau)\leq \mathfrak{p}$, then there exists a \nbd\ base $\mathcal{F}$ at 0 with cardinality $<\mathfrak{p}$.
Let $A\subset G$ be an infinite pseudo-intersection of $\mathcal{F}$; then $A$ is a convergent sequence with limit point $0$.
On the other hand, we have known that $(G, \tau)$ has no non-trivial convergent sequence in the proof of Lemma \ref{le1}.
This contradiction yields that $\chi(G, \tau)>\mathfrak{p}$.

Now consider a $k$-maximal topology $\tau_k$. There exists an non-refinable chain of group topologies
$\tau_1\succ\tau_2\succ...\succ\tau_k$ with $\tau_i$ $i$-maximal, where $1\leq i\leq k$.
By the result of \cite[Theorem 4.5]{HPTX2}, we have the following inequality
$$\chi(G, \tau_1)\leq \chi(G, \tau_2)\leq ... \leq\chi(G, \tau_k).$$
Hence $\chi(G, \tau_k)\geq \mathfrak{p}$.
Using this result and Lemmas \ref{le0} and \ref{le2}, one obtains a direct proof that $G$ does not satisfies $P_k$ for $k\geq 2$.
\end{remark}
\section{Questions and Comments}

\begin{question}\label{q2}
Does the lattice $\mathcal{L}_G$ remain semi-modular or dually semi-modular for solvable groups $G$?
\end{question}

Considering lattices with the \emph{(dual) Birkhoff property} – which implies the Weak Jordan-H\"older Chain Condition for finite chains – we ask:

\begin{question}\label{q3}
For solvable $G$, does $\mathcal{L}_G$ satisfy the (dual) Birkhoff property?
\end{question}
For an infinite abelian group $G$, we define $\mathcal{P}_G$ as the subset of $\mathcal{L}_G$ comprising all meets of families of maximal group topologies (recalling that maximal group topologies are coatoms in $\mathcal{L}_G$ and remain Hausdorff). Formally, this constitutes the collection
\[
\{\tau_{\mathcal{B}} := \bigwedge \mathcal{B} : \mathcal{B} \subseteq \mathcal{A}_1(G)\},
\]
where $\mathcal{A}_1(G)$ denotes the set of maximal group topologies, with $\bigwedge \emptyset$ interpreted as the discrete topology. 

For any $\mathcal{B} \subseteq \mathcal{A}_1(G)$, let $\overline{\mathcal{B}} := \{\tau \in \mathcal{A}_1(G) : \bigwedge \mathcal{B} \leq \tau\}$. Observe that $\tau_{\mathcal{B}} = \tau_{\overline{\mathcal{B}}}$ holds by construction.
 Crucially, any pair $\tau_{\mathcal{B}}, \tau_{\mathcal{C}} \in \mathcal{P}_G$ admits a least upper bound in $\mathcal{P}_G$ given by $\tau_{\overline{\mathcal{B}} \cap \overline{\mathcal{C}}}$, denoted $\tau_{\mathcal{B}} \vee_p \tau_{\mathcal{C}}$. Since $\mathcal{P}_G$ forms a $\wedge$-subsemilattice of $\mathcal{L}_G$, equipping it with the operations
\[
\tau_{\mathcal{B}} \wedge_p \tau_{\mathcal{C}} := \tau_{\mathcal{B}} \wedge \tau_{\mathcal{C}},
\]
we obtain a lattice structure $(\mathcal{P}_G, \vee_p, \wedge_p)$, which we term the \emph{Prodanov lattice} in honor of Iv. Prodanov. His pioneering work on submaximal group topologies \cite{Pro} identified these as the bottom elements of $\mathcal{P}_G$ – the infima of all maximal group topologies. Notably, Prodanov demonstrated that the submaximal topology refines the Bohr topology (i.e., the finest precompact topology) through explicit neighborhood basis construction at the identity. This result establishes that all elements of $\mathcal{P}_G$ are necessarily Hausdorff.

\begin{question}\label{QPro}
For an infinite abelian group $G$:
\begin{itemize}
    \item[(i)] Is $\mathcal{P}_G$ a sublattice of $\mathcal{L}_G$?
    \item[(ii)] Must every $k$-maximal group topology ($k > 1$) belong to $\mathcal{P}_G$? Equivalently, is every $k$-maximal topology expressible as the meet of $k$ maximal topologies?
    \item[(iii)] Do all maximal chains in $\mathcal{P}_G$ share equal cardinality?
    \item[(iv)] What is the supremum of cardinalities of chains in $\mathcal{P}_G$?
    \item[(v)] Is $\mathcal{P}_G$ freely generated by $\mathcal{A}_1(G)$? Specifically, does $\overline{\mathcal{B}} = \mathcal{B}$ hold for all $\mathcal{B} \subseteq \mathcal{A}_1(G)$?
\end{itemize}
\end{question}

We remark that condition (v) fails generically. Consider isomorphic infinite abelian groups $H_1, H_2$ with $K = H_1 \oplus H_2$, and let $H_3 = \{a + \varphi(a) : a \in H_1\}$ for an isomorphism $\varphi: H_1 \to H_2$. Given maximal topologies $\lambda_i$ on $H_i$ ($i = 1,2,3$), their extensions $\tau_i$ to $K$ satisfy:
\[
\tau_1 \wedge \tau_2 = \tau_2 \wedge \tau_3 = \tau_3 \wedge \tau_1,
\]
with topological basis $\{U + V : U \in \tau_1, V \in \tau_2\}$. Consequently, 
\[
\overline{\{\tau_1, \tau_2\}} = \overline{\{\tau_2, \tau_3\}} = \overline{\{\tau_3, \tau_1\}},
\]
invalidating (v) for $K$. For any abelian group $G$ containing $K$, the lattice embedding $\mathcal{L}_K \hookrightarrow \mathcal{L}_G$ via $\sigma \mapsto \sigma'$ (where $\sigma'$ extends $\sigma$ with $K$ open) preserves this failure. Thus, groups satisfying (v) cannot contain $\bigoplus_\omega \mathbb{Z}(p)$ or $\mathbb{Z} \oplus \mathbb{Z}$, implying finite $p$-rank and free-rank $\leq 1$. This motivates:

\begin{question}
Does $\mathbb{Z}$ satisfy condition (v) in Question \ref{QPro}?
\end{question}



\end{document}